\documentclass[a4paper,11pt]{amsart}
\usepackage{wrapfig}
\usepackage[dvips]{graphicx}
\usepackage{amsmath,amsthm,amssymb,amscd}
\theoremstyle{definition}
\newtheorem{thm}{Theorem}[section]

\newtheorem{pro}[thm]{Proposition}
\newtheorem{cor}[thm]{Corollary}
\newtheorem{lem}[thm]{Lemma}
\newtheorem{ex}[thm]{Example}
\newtheorem{rem}[thm]{Remark}
\theoremstyle{definition}

\begin{document}

\title{Fundamental group of simple $C^*$-algebras with unique trace II}
\author{Norio Nawata}
\address[Norio Nawata]{Graduate School of Mathematics, 
Kyushu University, Motooka, 
Fukuoka, 819-0395,  Japan}      

\author{Yasuo Watatani}
\address[Yasuo Watatani]{Department of Mathematical Sciences, 
Kyushu University, Motooka, 
Fukuoka, 819-0395,  Japan}
\maketitle
\begin{abstract}
We show that any countable  subgroup of the multiplicative group 
$\mathbb{R}_+^{\times}$ of positive real numbers can be realized 
as the  fundamental group $\mathcal{F}(A)$  of a separable 
simple unital $C^*$-algebra $A$ with unique trace. 
Furthermore for any fixed countable subgroup $G$ of $\mathbb{R}_+^{\times}$, 
there exist uncountably many mutually nonisomorphic such algebras $A$ with 
$G = \mathcal{F}(A)$. 
\end{abstract}
\section{Introduction}
Let $M$ be a factor of type $II_1$ with a normalized trace 
$\tau$. Murray and von Neumann  introduced 
the fundamental group ${\mathcal F}(M)$ of $M$ in \cite{MN}. 
The fundamental group ${\mathcal F}(M)$ of $M$ is a 
subgroup of the multiplicative group 
$\mathbb{R}_+^{\times}$ of positive real numbers. 
They showed that if $M$ is  hyperfinite, then 
${\mathcal F}(M) = {\mathbb R_+^{\times}}$. 
In our previous paper \cite{NW}, 
we introduced the fundamental group $\mathcal{F}(A)$ 
of a simple unital $C^*$-algebra $A$ with a unique normalized trace $\tau$
based on the computation  of Picard groups by 
Kodaka \cite{kod1}, \cite{kod2}, \cite{kod3}. 
We compute the fundamental groups of several nuclear or nonnuclear 
$C^*$-algebras.   
$K$-theoretical obstruction enable us to compute the fundamental 
group easily. 

There has been many works on the computation of fundamental groups of the factors 
of type $II_1$. 
Voiculescu \cite{Vo} showed that the fundamental group 
${\mathcal F}(L(\mathbb{F}_{\infty}))$ 
of the group factor $L(\mathbb{F}_{\infty})$ 
of the free group $\mathbb{F}_{\infty}$ contains the positive rationals and 
Radulescu proved that 
${\mathcal F}(L(\mathbb{F}_{\infty})) = {\mathbb R}_+^{\times}$ in 
\cite{Ra}.  Connes \cite{Co} showed that if $G$ is an countable 
ICC group with property (T), then  ${\mathcal F}(L(G))$ is a countable group. 
Recently, Popa 
showed that any countable subgroup of $\mathbb R_+^{\times}$ 
can be realized as the fundamental group of some 
factor of type $II_1$ with separable predual in \cite{Po1}. 
Furthermore Popa and Vaes \cite{PV} exhibited a large family $\mathcal{S}$ 
of subgroups of $\mathbb{R}_{+}^\times$, containing $\mathbb{R}_{+}^\times$ 
itself, all of its countable subgroups, as well as uncountable subgroups with 
any Haussdorff dimension in $(0,1)$, such that for each $G\in\mathcal{S}$ 
there exist many free ergodic measure preserving actions of $\mathbb{F}_{\infty}$ 
for which the associated $II_1$ factor $M$ has fundamental group equal to $G$. 

In this paper we show that any countable subgroup of $\mathbb{R}_+^{\times}$ 
can be realized as the fundamental group of a separable simple unital 
$C^*$-algebra with unique trace. 
Furthermore for any fixed countable subgroup $G$ of $\mathbb{R}_+^{\times}$, 
there exist uncountably many mutually nonisomorphic such algebras $A$ with 
$G = \mathcal{F}(A)$. 
We apply a method of Blackadar \cite{Bla} and Phillips 
\cite{Phi} to the  type $II_1$ factors of  Popa \cite{Po1}. 
Our new examples  are nonnuclear. 

On the other hand, for an additive subgroup $E$ of 
$\mathbb{R}$ containing 1, 
we define the positive inner multiplier group $IM_+(E)$ of $E$  by 
$$
IM_+(E) = \{t \in {\mathbb R}_+^{\times} \ | t \in E, t^{-1}  \in E, \text{ and } 
        tE = E \}. 
$$
Then we have $\mathcal{F}(A) \subset IM_+(\tau_*(K_0(A)))$. 
Almost all examples provided in \cite{NW} 
satisfy  $\mathcal{F}(A)=IM_+(\tau_*(K_0(A)))$. 
We should note  that not all countable subgroups of $\mathbb{R}_{+}^{\times}$ 
arise as  $IM_+(E)$. For example, 
$\{9^n \in \mathbb{R}_{+}^{\times} \ | 
n \in {\mathbb Z} \}$ does not arise as $IM_+(E)$ 
for any additive subgroup $E$ of $\mathbb{R}$ containing 1. Therefore 
if the fundamental group a $C^*$-algebra $A$ is equal to 
$\{9^n \in \mathbb{R}_{+}^{\times} \ | 
n \in {\mathbb Z} \}$ and $A$ is in a classifiable class  by 
the Elliott invariant, then  
$\tau_* : K_0(A) \rightarrow \tau_*(K_0(A))$ cannot  be an 
order isomorphism. Matui informed us that there exists such an AF-algebra.

\section{Hilbert $C^*$-modules and Picard groups} 
We recall some definitions and notations in \cite{NW}. 
Let $A$ be a simple unital $C^*$-algebra with a unique normalized trace $\tau$ 
and 
$\mathcal{X}$ a right Hilbert $A$-module. 
(See \cite{Lan}, \cite{MT} for the basic facts on Hilbert modules.)  
We denote by $L_A(\mathcal{X})$ 
the algebra of the adjointable operators on $\mathcal{X}$. 
For $\xi,\eta \in \mathcal{X}$, a  "rank one operator" $\Theta_{\xi,\eta}$
is defined by $\Theta_{\xi,\eta}(\zeta) 
= \xi \langle\eta,\zeta\rangle_A$ for $\zeta \in \mathcal{X}$. 
We denote by $K_A(\mathcal{X})$ the closure 
of the linear span of "rank one operators" $\Theta_{\xi,\eta}$.
We call a finite 
set $\{\xi_i\}_{i=1}^n\subseteq \mathcal{X}$ a {\it finite basis} of $\mathcal{X}$ if 
$\eta =\sum_{i=1}^n\xi_i\langle\xi_i,\eta\rangle_A$ for any $\eta\in\mathcal{X}$, see \cite{KW}, \cite{W}. 
It is also called a frame as in \cite{FL}.   
If $A$ is unital and there exists a finite basis for  $\mathcal{X}$, then 
$L_A(\mathcal{X})=K_A(\mathcal{X})$. 
Let $\mathcal{H}(A)$ denote the 
set of isomorphic classes $[\mathcal{\mathcal{X}}]$ of 
right Hilbert $A$-modules $\mathcal{X}$ with finite basis. 

Let $B$ be a $C^*$algebra. 
An $A$-$B$-equivalence bimodule is an $A$-$B$-bimodule $\mathcal{F}$ which is simultaneously a 
full left Hilbert $A$-module under a left $A$-valued inner product $_A\langle\cdot ,\cdot\rangle$ 
and a full right Hilbert $B$-module under a right $B$-valued inner product $\langle\cdot ,\cdot\rangle_B$, 
satisfying $_A\langle\xi ,\eta\rangle\zeta =\xi\langle\eta ,\zeta\rangle_B$ for any 
$\xi, \eta, \zeta \in \mathcal{F}$. We say that $A$ is {\it Morita equivalent} to $B$ 
if there exists an $A$-$B$-equivalence bimodule. 
The dual module $\mathcal{F}^*$ of an $A$-$B$-equivalence bimodule $\mathcal{F}$ is a set 
$\{\xi^* ;\xi\in\mathcal{F} \}$ with the operations such that $\xi^* +\eta^*=(\xi +\eta )^*$, 
$\lambda\xi ^*=(\overline{\lambda}\xi)^*$, $b\xi^* a=(a^*\xi b^*)^*$, 
$_B\langle\xi^*,\eta^*\rangle =\langle\eta ,\xi\rangle_B$ and 
$\langle \xi^*,\eta^*\rangle_A =\;_A\langle\eta ,\xi\rangle$. 
Then $\mathcal{F}^*$ is a $B$-$A$-equivalence bimodule. 
We refer the reader to \cite{RW},\cite{R2} for the basic facts on 
equivalence bimodules and Morita equivalence. 

We review elementary facts on the Picard groups of 
$C^*$-algebras introduced by Brown, Green and Rieffel 
in \cite{BGR}. 
For $A$-$A$-equivalence bimodules 
$\mathcal{E}_1$ and 
$\mathcal{E}_2$, we say that $\mathcal{E}_1$ is isomorphic to $\mathcal{E}_2$ as an equivalence 
bimodule if there exists a $\mathbb{C}$-liner one-to-one map $\Phi$ of $\mathcal{E}_1$ onto 
$\mathcal{E}_2$ with the properties such that $\Phi (a\xi b)=a\Phi (\xi )b$, 
$_A\langle \Phi (\xi ) ,\Phi(\eta )\rangle =\;_A\langle \xi ,\eta\rangle$ and 
$\langle \Phi (\xi ) ,\Phi(\eta )\rangle_A =\langle\xi,\eta\rangle_A$ for $a,b\in A$, 
$\xi ,\eta\in\mathcal{E}_1$. 
The set of isomorphic classes $[\mathcal{E}]$ of the $A$-$A$-equivalence 
bimodules $\mathcal{E}$ forms a group under the product defined by 
$[\mathcal{E}_1][\mathcal{E}_2] = [\mathcal{E}_1 \otimes_A \mathcal{E}_2]$. 
We call it the {\it Picard group} of $A$ and denote it  by  $\mathrm{Pic}(A)$. 
The identity of $\mathrm{Pic}(A)$ is given by 
the $A$-$A$-bimodule $\mathcal{E}:= A$ with  
$\; _A\langle a_1 ,a_2 \rangle = a_1a_2^*$ and $\langle a_1 ,a_2\rangle_A = a_1^*a_2$ for 
$a_1,a_2 \in A$. The inverse element of $[\mathcal{E}]$ in the Picard group of $A$ 
is the dual module $[\mathcal{E}^*]$. 
Let $\alpha$ be an automorphism of $A$, and let 
$\mathcal{E}_{\alpha}^A=A$ with the obvious left $A$-action and the obvious $A$-valued inner product. 
We define the right $A$-action on $\mathcal{E}_\alpha^A$ by 
$\xi\cdot a=\xi\alpha(a)$ for 
any $\xi\in\mathcal{E}_\alpha^A$ and $a\in A$, and the right $A$-valued inner product by 
$\langle\xi ,\eta\rangle_A=\alpha^{-1} (\xi^*\eta)$ for any $\xi ,\eta\in\mathcal{E}_\alpha^A$.
Then $\mathcal{E}_{\alpha}^A$ is an $A$-$A$-equivalence bimodule. For $\alpha, \beta\in\mathrm{Aut}(A)$, 
$\mathcal{E}_\alpha^A$ is isomorphic to $\mathcal{E}_\beta^A$ if and only if 
there exists a unitary $u \in A$ such that 
$\alpha = ad \ u \circ \beta $. Moreover, ${\mathcal E}_\alpha^A \otimes 
{\mathcal E}_\beta^A$ is 
isomorphic to $\mathcal{E}_{\alpha\circ\beta}^A$. Hence we obtain an homomorphism $\rho_A$ 
of $\mathrm{Out}(A)$ to $\mathrm{Pic}(A)$. 
An $A$-$B$-equivalence bimodule $\mathcal{F}$ induces an isomorphism $\Psi$ 
of $\mathrm{Pic}(A)$ to $\mathrm{Pic}(B)$ by 
$\Psi ([\mathcal{E}])=[\mathcal{F}^*\otimes\mathcal{E}\otimes\mathcal{F}]$ 
for $[\mathcal{E}]\in\mathrm{Pic}(A)$. 
Therefore if $A$ is Morita equivalent to $B$, then $\mathrm{Pic}(A)$ is isomorphic to 
$\mathrm{Pic}(B)$. 
Since $A$ is unital, any 
$A$-$A$-equivalence bimodule is a finitely generated projective $A$-module as a right 
module with a finite basis $\{\xi_i\}_{i=1}^n$. 
Put  $p=(\langle\xi_i,\xi_j\rangle_A)_{ij} \in M_n(A)$. 
Then $p$ is a projection and $\mathcal{E}$ is isomorphic to 
$pA^n$ as a right Hilbert $A$-module 
with an isomorphism  of $A$ to  $pM_n(A)p$ as a $C^*$-algebra.

Define a map $\hat{T}_A : \mathcal{H}(A) \rightarrow 
\mathbb{R}_{+}$ by 
$\hat{T}_A([\mathcal{X}])=\sum_{i=1}^n\tau (\langle\xi_i,\xi_i\rangle_A)$, 
where $\{\xi_i\}_{i=1}^n$ is a finite basis of $\mathcal{X}$. 
Then $\hat{T}_A([\mathcal{X}])$  
does not depend on the choice of basis and $\hat{T}_A$ is well-defined. 
We can define a map $T_A$ of $\mathrm{Pic}(A)$ to $\mathbb{R}_{+}$ 
by the same way of $\hat{T}_A$. 
We showed that $T_A$ is a multiplicative map 
and $T_A(\mathcal{E}_{id}^A) = 1$ in \cite{NW}. 
Moreover, we can show the following proposition by a similar argument 
in the proof of Proposition 2.1 in \cite{NW}. 

\begin{pro}\label{pro:multiplicative}
Let $A$ and $B$ be simple unital $C^*$-algebras with unique trace. 
Assume that $\mathcal{F}$ is an $A$-$B$-equivalence bimodule and 
$\mathcal{X}$ is a right Hilbert $A$-module. Then 
\[\hat{T}_B([\mathcal{X}\otimes\mathcal{F}])=
   \hat{T}_A([\mathcal{X}])\hat{T}_B([\mathcal{F}]).\]
\end{pro}

We denote by $Tr$ the usual unnormalized trace on $M_n(\mathbb{C})$. 
 Put 
$$
\mathcal{F}(A) :=\{ \tau\otimes Tr(p) \in \mathbb{R}^{\times}_{+}\ | \ 
 p \text{ is a projection in } M_n(A) \text{ such that } pM_n(A)p  \cong A \}. 
$$
Then $\mathcal{F}(A)$ is equal to 
the image of $T_A$ and 
a multiplicative subgroup of $\mathbb{R}^{\times}_{+}$ by Theorem 3.1 in \cite{NW}. 
We call 
$\mathcal{F}(A)$ the {\it fundamental group} of $A$. 
If $A$ is separable, then $\mathcal{F}(A)$ is countable. 
We shall show that the fundamental group is a Morita equivalence invariant for 
simple unital $C^*$-algebras with unique trace.

\begin{pro}
Let $A$ and $B$ be simple unital $C^*$-algebras with unique trace. 
If $A$ is Morita equivalent to $B$, then $\mathcal{F}(A)=\mathcal{F}(B)$. 
\end{pro}
\begin{proof}
By assumption, there exists an $A$-$B$-equivalence bimodule $\mathcal{F}$, 
and $\mathcal{F}$ induces an isomorphism 
$\Psi$ of $\mathrm{Pic}(A)$ to $\mathrm{Pic}(B)$ such that 
$\Psi ([\mathcal{E}])=[\mathcal{F}^*\otimes\mathcal{E}\otimes\mathcal{F}]$ for 
$[\mathcal{E}]\in\mathrm{Pic}(A)$. 
Since $\mathcal{F}^*\otimes\mathcal{F}$ is isomorphic to $\mathcal{E}_{id}^B$, 
Proposition \ref{pro:multiplicative} implies 
\[\hat{T}_A([\mathcal{F}^*])\hat{T}_B([\mathcal{F}])=
T_B([\mathcal{F}^*\otimes\mathcal{F}])=1.\] 

For $[\mathcal{E}]\in \mathrm{Pic}(A)$, 
\[T_B([\mathcal{F}^*\otimes\mathcal{E}\otimes\mathcal{F}])=
\hat{T}_A([\mathcal{F}^*])\hat{T}_B([\mathcal{E}\otimes\mathcal{F}])
=\hat{T}_A([\mathcal{F}^*])T_A([\mathcal{E}])\hat{T}_B([\mathcal{F}])\]

by Proposition \ref{pro:multiplicative}. 
Therefore $T_B([\Psi (\mathcal{E})])=T_A([\mathcal{E}])$ and 
$\mathcal{F}(A)=\mathcal{F}(B)$.
\end{proof}

\section{New examples}

An idea of our construction comes from the following results of 
Blackadar, Proposition 2.2 of \cite{Bla} and 
Phillips, Lemma 2.2 of \cite{Phi}. 
\begin{lem}[\cite{Bla}(Blackadar)]  
Let $M$ be a simple $C^*$-algebra, and let $A\subset M$ be a separable $C^*$-subalgebra. 
Then there exists a simple separable $C^*$-subalgebra $B$ with $A\subset B\subset M$. 
\end{lem}
\begin{lem}[\cite{Phi}(Phillips)]
Let $M$ be a unital $C^*$-algebra, and let $A\subset M$ be a separable $C^*$-subalgebra. 
Then there exists a separable $C^*$-subalgebra $B$ with $A\subset B\subset M$ such that 
every tracial state on $B$ is the restriction of a tracial state on $M$. 
\end{lem}
The following lemma is just a combination of  the two results above.  
\begin{lem}\label{lem:key}
Let $M$ be a simple $C^*$-algebra with unique trace $\hat{\tau}$, and let $A\subset M$ 
be a separable $C^*$-subalgebra. 
Then there exists a simple separable $C^*$-subalgebra $B$ with $A\subset B\subset M$
such that $B$ has a unique trace $\tau$ that is a restriction of $\hat{\tau}$. 
\end{lem}

\begin{thm}\label{thm:main}
Let $G$ be a countable subgroup of $\mathbb{R}_{+}^{\times}$. 
Then there exist uncountably many mutually nonisomorphic  separable  
simple nonnuclear unital $C^*$-algebras $A$
with unique trace such that the fundamental group $\mathcal{F}(A)=G$. 
\end{thm}
\begin{proof}
First we shall show that there exists a separable 
simple unital $C^*$-algebra $A$ 
with unique trace such that $\mathcal{F}(A)=G$. 
There exists a type $II_1$ factor $M$ with separable predual such that 
$\mathcal{F}(M)=G$, which is constructed by Popa \cite{Po1}. 
Let $S_1\subset M$ be a countable subset that is weak operator dense in $M$. 
We denote by $\hat{\tau}$ the unique trace of $M$. 
We enumerate the countable semigroup $G\cap (0,1]$ 
by $\{t_m:m\in\mathbb{N}\}$. 
Since $\mathcal{F}(M)=G$ and $M$ is a factor of type $II_1$, 
for any $m\in\mathbb{N}$
there exist a projection $p_m$ 
in $M$ such that $\hat{\tau}(p_m)=t_m$ and an isomorphism $\phi_m$ of $M$ onto $p_mMp_m$ 
. 
Define $B_0\subset M$ be the unital $C^*$-subalgebra 
of $M$ generated by $S_1$ and $\{p_m:m\in\mathbb{N}\}$. 
By Lemma \ref{lem:key}, 
there exists a separable simple unital 
$C^*$-algebra $A_0$ with a unique trace $\tau_0$ 
such that $B_0\subset A_0\subset M$. 
Let $B_1\subset M$ be the  $C^*$-subalgebra 
of $M$ generated by $A_0$, $\cup_{m\in\mathbb{N}}\phi_m(A_0)$ and 
$\cup_{m\in\mathbb{N}}\phi_m^{-1}(p_mA_0p_m)$. By the same way, 
there exists a separable simple unital $C^*$-algebra $A_1$ 
with a unique trace $\tau_1$ 
such that $B_1\subset A_1\subset M$.
We construct inductively  $C^*$-algebras $B_n \subset A_n \subset M$ 
as follows: 
Let $B_n\subset M$ be the  $C^*$-subalgebra 
of $M$ generated by $A_{n-1}$, $\cup_{m\in\mathbb{N}}\phi_m(A_{n-1})$ and 
$\cup_{m\in\mathbb{N}}\phi_m^{-1}(p_mA_{n-1}p_m)$. By Lemma \ref{lem:key}, 
there exists a separable simple unital $C^*$-algebra $A_n$ 
with a unique trace $\tau_n$ 
such that $B_n\subset A_n\subset M$. 
Then we have 
$$
B_0 \subset A_0\subset B_1\subset A_1\subset \dots 
B_n \subset A_n \dots    \subset M, 
$$
and 
$\phi_m(A_{n-1})\subset p_mA_np_m$ and $\phi_m^{-1}(p_mA_{n-1}p_m)\subset A_n
\text { for any } m \in\mathbb{N}$. 
Set $A=\overline{\cup_{n=0}^{\infty}A_n}$. Then $A$ is a  separable 
simple unital $C^*$-algebra 
$A$ with a unique trace $\tau$. 
By the construction, 
$\phi_m(A)=p_mAp_m$ for any $m\in\mathbb{N}$. Hence 
$G \subset \mathcal{F}(A)$. 
Since $\pi_\tau (A)''$ is isomorphic to $M$, 
$$
\mathcal{F}(A) \subset \mathcal{F}(\pi_\tau (A)'') 
= \mathcal{F}(M) = G
$$
by Proposition 3.29 of \cite{NW}. 
Thus $\mathcal{F}(A) = G$. Moreover $A$ is not nuclear, because 
 $A$ is  weak operator dense in a factor $M$ that is not hyperfinite.

Next we shall show that there exist uncountably many mutually nonisomorphic 
such examples. 
Let $E$ be a countable additive subgroup of $\mathbb{R}$. 
We enumerate by $\{r_m:\in\mathbb{N}\}$ the positive elements of $E$. 
Since $M$ is a factor of type $II_1$,  for any 
$m\in\mathbb{N}$ there exist a natural number $k$ and 
a projection $q_m\in M_k(M)$ such that $\hat{\tau}\otimes Tr(q_m)=r_m$. 
Define $S_2 \subset M $ to be the union of  
the matrix elements of $q_m$ for running  $m \in\mathbb{N}$. 
Let $C_0$ the $C^*$-subalgebra of $M$ 
generated by $S_2$ and $A$. 
By a similar argument as the first paragraph, 
we can construct  a separable 
simple unital $C^*$-algebra $C$ with unique trace such that 
$\mathcal{F}(C)=G$ and $C_0\subset C \subset M$. 
Then it is clear that $E$ is contained in $\tau_*(K_0(C))$. 
Since no countable union of 
countable subgroups of $\mathbb{R}$ can contain all countable subgroups of $\mathbb{R}$, 
we can construct uncountably many mutually nonisomorphic examples by the choice of $E$. 
\end{proof}
\begin{rem}
In fact, we show that there exist uncountably many Morita inequivalent 
separable simple nonnuclear unital $C^*$-algebras $A$ with unique trace 
such that the fundamental group $\mathcal{F}(A)=G$ in the proof above. 
\end{rem}
\begin{rem}\label{rem:main}
We can choose a $C^*$-algebra $A$ in the theorem above 
so that $A$ has stable rank one and 
real rank zero and 
$\tau_* : K_0(A) \rightarrow \tau_*(K_0(A))$ is an order isomorphism 
by using Lemma 2.3, Lemma 2.4 and Lemma 2.5 of \cite{Phi}. 
Then  we have the following exact sequence 
by Proposition 3.26 of \cite{NW}: 
\[\begin{CD}
      {1} @>>> \mathrm{Out}(A) @>\rho_A>> \mathrm{Pic}(A) @>T>> \mathcal{F}(A)
 @>>> {1} \end{CD}. \]
\end{rem}
\begin{rem}
 We do not know whether any countable subgroup 
of $\mathbb{R}_{+}^{\times}$ can be realized as the fundamental group 
of  a separable unital simple 
{\it nuclear}  $C^*$-algebra 
with unique trace. 
\end{rem}

\begin{lem}\label{lem:K0}
Let $M_1$ and $M_2$ be factors of type $II_1$, and let $A_0\subset M_1$ and $B_0\subset M_2$ 
be separable $C^*$-subalgebras. 
Then there exist separable simple unital $C^*$-algebras $A$ and $B$ with the unique 
traces $\tau_A$ and $\tau_B$ such that 
$A_0\subset A\subset M_1$, $B_0\subset B\subset M_2$ and $(\tau_{A})_*(K_0(A))=(\tau_{B})_*(K_0(B))$. 
\end{lem}
\begin{proof}
Let $\tau_1$ be the unique trace on $M_1$ and $\tau_2$ the unique trace on $M_2$. 
Since $A_0$ and $B_0$ are separable $C^*$-algebras, 
$(\tau_{1}|_{A_0})_{*}(K_0(A_0))$ and $(\tau_{2}|_{B_0})_{*}(K_0(B_0))$ are 
countable groups. 
We enumerate the positive elements of $(\tau_{1}|_{A_0})_{*}(K_0(A_0))$ by 
$\{t_{m}:m\in\mathbb{N}\}$ and the positive elements of $(\tau_{2}|_{B_0})_{*}(K_0(B_0))$ by 
$\{r_{m}:m\in\mathbb{N}\}$. 
Since $M_1$ and $M_2$ are factors of type $II_1$,  for any 
$m\in\mathbb{N}$ 
there exist a natural number $k$ and 
projections $p_{m}\in M_k (M_1)$ and $q_{m}\in M_k(M_2)$ such that 
$\tau_{1}\otimes Tr(p_{m})=r_{m}$ and $\tau_{2}\otimes Tr(q_{m})=t_{m}$. 
Put $S_1 \subset M_1 $ (resp. $S_2 \subset M_2$) to be the union of 
the matrix elements of $p_m$ (resp. $q_m$) for running  $m \in\mathbb{N}$. 
Define $C_1\subset M_1$ (resp. $D_1\subset M_2$) be the unital $C^*$-subalgebra 
of $M_1$ (resp. $M_2$) generated by $A_0$ and $S_1$ 
(resp. $B_0$ and $S_2$). 
By Lemma \ref{lem:key}, 
there exist separable simple unital 
$C^*$-algebras $A_1$ and $B_1$ with a unique trace such that 
$C_1\subset A_1\subset M_1$ and $D_1\subset B_1\subset M_2$. 
Then we have  
$(\tau_{1}|_{A_0})_{*}(K_0(A_0))\subset (\tau_{2}|_{B_1})_{*}(K_0(B_1))$ and 
$(\tau_{2}|_{B_0})_{*}(K_0(B_0))\subset (\tau_{1}|_{A_1})_{*}(K_0(A_1))$. 
In a similar way, we construct inductively simple separable unital $C^*$-algebras 
$A_n \subset M_1$ and $B_n \subset M_2$ with unique trace such that 
$(\tau_{1}|_{A_{n-1}})_{*}(K_0(A_{n-1}))\subset (\tau_{2}|_{B_n})_{*}(K_0(B_n))$ and 
$(\tau_{2}|_{B_{n-1}})_{*}(K_0(B_{n-1}))\subset (\tau_{1}|_{A_n})_{*}(K_0(A_n))$. 
Set $A=\overline{\cup_{n=1}^{\infty}A_n}$ and $B=\overline{\cup_{n=1}^{\infty}B_n}$. 
Then $A$ and $B$ are separable simple unital $C^*$-algebras with unique trace. 
We denote by $\tau_A$ the unique trace on $A$ and by $\tau_B$ the unique traces on $B$. 
By the construction,  
$(\tau_A)_{*}(K_0(A))=(\tau_B)_{*}(K_0(B))$. 
\end{proof}

We denote by $\mathrm{Ell}(A)$ the Elliott invariant $(K_0(A),K_0(A)_+,[1]_0,K_1(A))$. 
\begin{cor}
For any countable subgroups $G_1$ and $G_2$ of $\mathbb{R}_{+}^{\times}$, 
there exist separable simple nonnuclear unital $C^*$-algebras $A$ and $B$ 
with unique trace such that 
$\mathrm{Ell}(A)\cong\mathrm{Ell}(B)$, 
$\mathcal{F}(A)=G_1$ and $\mathcal{F}(B)=G_2$. 
\end{cor}
\begin{proof}
The proof of Theorem \ref{thm:main}, Lemma \ref{lem:K0} and Lemma 2.5 of \cite{Phi} 
implies that there exist separable simple nonnuclear unital $C^*$-algebras 
$A$ and $B$ with the unique traces $\tau_A$ and $\tau_B$ such that 
$(\tau_A)_* : K_0(A) \rightarrow (\tau_A)_*(K_0(A))$ and 
$(\tau_B)_* : K_0(B) \rightarrow (\tau_B)_*(K_0(B))$ are order isomorphisms, 
$\mathcal{F}(A)=G_1$, $\mathcal{F}(B)=G_2$, $K_1(A)=K_1(B)=0$ and $
(\tau_A)_{*}(K_0(A))=(\tau_B)_{*}(K_0(B))$. 
Since $(\tau_A)_*$ and $(\tau_B)_*$ are order isomorphisms and 
$(\tau_A)_{*}(K_0(A))=(\tau_B)_{*}(K_0(B))$, we see that $\mathrm{Ell}(A)\cong\mathrm{Ell}(B)$. 
\end{proof}

For a positive number $\lambda$, let 
$G_{\lambda} = \{ {\lambda}^n \in \mathbb{R}_{+}^{\times} \ | 
\ n \in \mathbb{Z} \}$ be the multiplicative subgroup of 
$\mathbb{R}_{+}^{\times}$ generated by $\lambda$. 
In the below we shall consider whether $G_{\lambda}$
can be realized as the fundamental group of a nuclear $C^*$-algebra. 

\begin{pro}
Let $\lambda$ be a prime number or a positive transcendental number. 
Then there exists a simple $AF$-algebra $A$ with unique trace such that 
$\mathcal{F}(A) = G_{\lambda}$. 
\end{pro}
\begin{proof}
Let $\lambda$ be a prime number. Consider a UHF-algebra 
$A = M_{{\lambda}^\infty}$. Then $\mathcal{F}(A) = G_{\lambda}$ as in 
Example 3.11 of \cite{NW}. Next we assume that $\lambda$ is a 
positive transcendental number. Let $R_{\lambda}$ 
be the  unital subring of $\mathbb{R}$ 
generated by $\lambda$. Then the set $(R_{\lambda})^\times_{+}$ of positive 
invertible elements in $R_{\lambda}$ is equal to $G_{\lambda}$. The proof of 
Theorem 3.14  of \cite{NW} shows that there exists a simple unital $AF$-algebra $A$ 
with unique trace such that $\mathcal{F}(A) = G_{\lambda}$. 
\end{proof} 
Let $\mathcal{O}$ be an order of a real quadratic field or 
a real cubic field with one real embedding. 
Then $\mathcal{O}^{\times}_{+}=G_{\lambda}$ is singly 
generated and the generator $\lambda >1$ is called 
the fundamental unit of $\mathcal{O}$ by Dirichlet's unit theorem. 
We refer the reader to \cite{Neu} for details. 
The proof of Theorem 3.14  of \cite{NW} implies the following proposition. 
\begin{pro}
Let $\lambda$ be a fundamental unit of an order of a real quadratic field 
or a cubic field with one real embedding. 
Then there exists a simple $AF$-algebra $A$ with unique trace such that 
$\mathcal{F}(A) = G_{\lambda}$. 
\end{pro}

Note that if $p$ is a prime number and $n \geq 2$, then 
the subgroup $G_{\lambda}$ of $R_+^{\times}$ 
generated by ${\lambda} = p^n$ can not be 
the  positive inner multiplier group $IM_+(E)$ for any 
additive subgroup $E$ of 
$\mathbb{R}$ containing 1.  In fact, on the contrary, suppose  that 
$G_{\lambda} = IM_+(E)$ for some $E$. Then there exists a unital subring 
$R$ of $\mathbb{R}$ such that $G_{\lambda}= R_+^{\times}$ by 
Lemma 3.6 of \cite{NW}. Then 
$\frac{1}{p} = \frac{1}{\lambda} + \dots + \frac{1}{\lambda} 
\in R_+^{\times}$. This contradicts that $\frac{1}{p} \notin G_{\lambda}$.   
However,  we have another construction.  

\begin{ex}\label{ex:matui} For $\lambda = 3^2 = 9$, 
Matui shows us the following example: 
Let $A$ be an $AF$-algebra such that 
$$
K_0(A)=\{(\frac{b}{9^a},c) \in \mathbb{R} \times \mathbb{Z} \ | 
\ a,b,c\in\mathbb{Z},b\equiv c\; \mathrm{mod}\; 8\},
$$
$$
K_0(A)_{+}=\{(\frac{b}{9^a},c)\in K_0(A):\frac{b}{9^a}>0\}\cup \{(0,0)\}
\ \ \text{and} \ \   [1_A]_0=(1,1).
$$ 
Then 
$$
\mathcal{F}(A) = G_9:= 
 \{ 9^n \in \mathbb{R}_{+}^{\times} \ | 
\ n \in \mathbb{Z} \}
$$ 
Moreover $\tau_* : K_0(A) \rightarrow \tau_*(K_0(A))$ is not an 
order isomorphism and 
$\mathcal{F}(A) \not= IM_+(\tau_*(K_0(A)))$. 

Furthermore  Katsura suggests us the following examples: 
Let $\lambda = p^n$ for a prime number $p$  and a natural number $n \geq 2$. 
Then there exists a simple $AF$-algebra $A$ with unique trace such that 
$\mathcal{F}(A) = G_{\lambda}$.

First consider the case that $\lambda \geq 8$. 
Define 
$$
E=\{(\frac{b}{p^{na}},c) \in \mathbb{R} \times \mathbb{Z} 
\ | \ a,b,c\in\mathbb{Z},b\equiv c\; \mathrm{mod}\; (p^n-1) \}
$$ 
$$
E_+=\{(\frac{b}{p^{na}},c)\in E:\frac{b}{p^{na}}>0\}\cup \{(0,0)\}
\ \ \text{and} \ \  [u]_0=(1,1).
$$
Then there exists a simple $AF$-algebra $A$ 
such that 
$(K_0(A),K_0(A)_+,[1_A]_0)=(E,E_+,u)$ by \cite{EHS}. 
The classification theorem of \cite{E} and some computation yield that  
$\mathcal{F}(A) = G_{\lambda}$.

Next consider the case that $\lambda = 2^2=4$. 
Let 
$$
E=\{(\frac{b}{16^{a}},c) \in \mathbb{R} \times \mathbb{Z} \ | \ 
 a,b,c\in\mathbb{Z},b\equiv c\; \mathrm{mod}\; 5 \}
$$  
$$
E_+=\{(\frac{b}{16^a},c)\in E:\frac{b}{16^{a}}>0\}\cup \{(0,0)\}
\ \ \text{and} \ \   [u]_0=(1,1).
$$
Consider a simple $AF$-algebra $A$ such that 
$(K_0(A),K_0(A)_+,[1_A]_0)=(E,E_+,u)$. Then $\mathcal{F}(A) = G_{4}$.

\end{ex}

\end{document}